\documentclass[reqno,11pt,b5paper]{amsart}
\usepackage{latexsym,graphicx,verbatim,color,cleveref,enumerate,fontenc,setspace,cite}
\usepackage{caption,amsthm,mathrsfs}
\usepackage{amsmath,amssymb}
\usepackage[utf8]{inputenc} % per le lettere accentate in italiano
\usepackage{subcaption}
\usepackage{enumerate}
\usepackage{graphicx}
\usepackage[mathscr]{eucal}
\newcommand{\mau}{\geq}
\newcommand{\miu}{\leq}
\newcommand{\ep}{\varepsilon}
\newcommand{\N}{\mathbb{N}}
\newcommand{\R}{\mathbb{R}}
\newcommand{\Z}{\mathbb{Z}}
\newcommand{\disp}{\displaystyle}
\newcommand{\be}{\begin{equation}}
\newcommand{\ee}{\end{equation}}

\newcommand{\K}{K^{(M)}} % max-product operators
\newcommand{\V}{\bigvee}
\newcommand{\phis}{\phi_{\sigma}}

\newcommand{\acr}{\newline\indent}

\newtheorem{definition}{Definition}[section]

\newtheorem{theorem}[definition]{Theorem}
\newtheorem{lemma}[definition]{Lemma}

\begin{document}
\title[Approximation results in Orlicz spaces...]{Approximation results in Orlicz spaces\\ for sequences of Kantorovich max-product\\ neural network operators}
         
\author[D. Costarelli, A.R. Sambucini]{{\bf Danilo Costarelli}  and  {\bf Anna Rita Sambucini}}

\date{}
\address{\llap{*\,}Department of Mathematics and Computer Sciences\acr
1, Via Vanvitelli  \acr 06123-I  Perugia \acr (ITALY)}
\email{danilo.costarelli@unipg.it, anna.sambucini@unipg.it}
\thanks{This work was supported by   University of Perugia - Department of Mathematics and Computer Sciences - Grant Nr 2010.011.0403  and by the Grant   N.  U UFMBAZ2017/0000326 of GNAMPA - INDAM (Italy).\acr
Danilo Costarelli orcid ID: 0000-0001-8834-8877, 
Anna Rita Sambucini orcid ID: 0000-0003-0161-8729}


\begin{abstract}
In this paper we study the theory of the so-called Kantorovich max-product neural network operators in the setting of Orlicz spaces $L^{\varphi}$. The results here proved, extend those given by the authors in Result Math., 2016, to a more general context. The main advantage in studying neural network type operators in Orlicz spaces relies in the possibility to approximate not necessarily continuous functions (data) belonging to different function spaces by a unique general approach. Further, in order to derive quantitative estimates in this context, we introduce a suitable K-functional in $L^{\varphi}$ and use it to provide an upper bound for the approximation error of the above operators. Finally, examples of sigmoidal activation functions have been considered and studied.
\end{abstract}
\maketitle 
%==================================
\section{Introduction}

   The max-product version of several families of linear operators has been recently studied in many papers, see e.g., \cite{COGA1,COGA2,COGA3,COGA4}.

The main advantage provided by the max-product approach, firstly introduced by Coroianu and Gal (for a very complete overview see the monograph \cite{BECOGA1}), is that, it allows to obtain sharper approximation with respect to the corresponding linear counterparts.

   Based on the above motivations, the max-product version of the Kantorovich neural network (NN) operators has been introduced in \cite{COVI3}, for approximating functions of one-variable. An extension of this theory to the multivariate case has been given in \cite{COVI5}.
   
   The theory of the NN operators has been widely studied in last decade in view of its relation with the well-known theory of the feed-forward artificial neural networks, see \cite{ANCOGA1,CACH1,CACH2,CALIPA1}. Neural networks have been widely used in many applications fields: such as in artificial intelligence (\cite{BBM1,GOLO1,LLZ1,LWZ1,SXJ1}), neuro sciences and in approximation theory (\cite{MAI1,RIRU1,BAGR1,SS1}).  
   
     In particular, Kantorovich-type operators have been successfully used in order to approximate/reconstruct not necessarily continuous data (see e.g., \cite{BAMA1,COVI2}), and consequently they revealed to be suitable for applications to image processing, see e.g., \cite{ING1}.
   
   For all the above reasons, we have introduced and studied the Kantorovich max-product NN operators in \cite{COVI3}, obtaining e.g. convergence results in the $L^p$-spaces, $1 \miu p < +\infty$, other than pointwise and uniform approximation results in the space of continuous functions. Moreover, some quantitative estimates for the order of approximation in both the above settings have been achieved in \cite{COVI7}.
   
   In the present paper, we provide an extension of the results proved in both \cite{COVI3} and \cite{COVI7} in the more general context of Orlicz spaces $L^\varphi$, where $\varphi$ is a convex $\varphi$-function (\cite{MUORL,MU1}).
   
   The Orlicz spaces are general spaces, which include as particular case, the $L^p$-spaces, and many others, such as the interpolation and the exponential spaces.
   
   In particular, here we prove a modular convergence results for sequences of Kantorovich max-product NN operators, in case of functions belonging to $L^\varphi$. The modular convergence is the most natural when approximation processes in Orlicz spaces are studied. To reach this result, firstly we need to test the modular convergence in case of continuous functions, and to establish a modular inequality for the above nonlinear operators. 
   
   Furthermore, in order to achieve quantitative estimates in $L^\varphi$, we introduce an extension of the usual Peetre K-functionals (\cite{PE1}), in this general context.
   
   Such K-functionals have been used in order to establish a sharp upper bound for the modular of the aliasing errors provided by the Kantorovich max-product NN operators.
   
   In conclusion, we recall that NN-type operators are usually based on certain density functions generated by sigmoidal functions. The latter fact is typical when one deal with NN type approximation processes, due to the biological reasons for which the neural networks have been introduced, see \cite{LIV1,OL1}. Examples of sigmoidal functions (see Section \ref{sec3}) are given, e.g., by the logistic function, the hyperbolic tangent function, and many others (see also \cite{COSP1}).

%%%%%%%%%%%%%%%%%%%%%%%

\section{The Orlicz spaces} \label{sec2}

First of all, we need to introduce the general setting of Orlicz spaces, in which we will study the above neural network (NN) type operators.

   A function $\varphi: \R^+_0 \to \R^+_0$ which satisfies the following assumptions:
\begin{itemize}
	\item[$(\varphi 1)$] $\varphi \left(0\right)=0$, $\varphi \left(u\right)>0$ for every $u>0$;
	\item[$(\varphi 2)$] $\varphi$ is continuous and non decreasing on $\R^+_0$;
	\item[$(\varphi 3)$] $\disp \lim_{u\to +\infty}\varphi(u)\ =\ + \infty$,
\end{itemize}
is said to be a $\varphi$-function. Let $M([a,b])$ denotes the set of all (Lebesgue) measurable functions $f:[a,b] \to \R$, where $[a,b]$ denotes any bounded real interval. For any fixed $\varphi$-function $\varphi$, it is possible to define the functional $I^{\varphi} : M([a,b])\to [0,+\infty]$, as follows:
$$
I^{\varphi} \left[f\right] := \int_{a}^b \varphi(\left| f(x) \right|)\ dx,\ \hskip0.5cm f \in M([a,b]).
$$
The functional $I^\varphi$ is called a {\em modular} on $M([a,b])$. Then, we can define the {\em Orlicz space} generated by $\varphi$ by the set:
$$
L^{\varphi}([a,b]) := \left\{f \in M\left([a,b]\right):\ I^{\varphi}[\lambda f]<+\infty,\ \mbox{for\ some}\ \lambda>0\right\}.
$$
Moreover, we can also define the following useful subspace of $L^{\varphi}([a,b])$, the so-called space of the {\em finite elements} of $L^{\varphi}([a,b])$:
$$
E^{\varphi}([a,b]) := \left\{f \in M\left([a,b]\right):\ I^{\varphi}[\lambda f]<+\infty,\ \mbox{for\ every}\ \lambda>0\right\}.
$$
Now, exploiting the definition of the modular functional $I^{\varphi}$, it is possible to introduce a notion of convergence in $L^{\varphi}([a,b])$, the so-called {\em modular convergence}.

More precisely, a net of functions $(f_w)_{w>0} \subset L^{\varphi}([a,b])$ is said modularly convergent to a function $f \in L^{\varphi}([a,b])$ if:
\be
\lim_{w \to +\infty} I^{\varphi}\left[\lambda(f_w-f)\right]\ =\ 0,
\ee
for some $\lambda>0$. For a sake of completeness, we must recall that in $L^{\varphi}([a,b])$ it is also possible to define a strong notion of convergence, i.e., the so-called {\em (Luxemburg) norm-convergence}. 

  Namely, a net of functions $(f_w)_{w>0} \subset L^{\varphi}([a,b])$ is said norm convergent to a function $f \in L^{\varphi}([a,b])$ if:
\be
\lim_{w \to +\infty} I^{\varphi}\left[\lambda(f_w-f)\right]\ =\ 0,
\ee 
for every $\lambda>0$. It is possible to prove that, the norm and the modular convergence are equivalent if and only if the $\varphi$-function generating the Orlicz space satisfies the so-called $\Delta_2$-condition, i.e., the inequality:
\be \label{delta2}
\varphi(2u)\ \miu\ M\varphi(u), \hskip1.2cm u \in \R^+_0,
\ee 
for some $M>0$.

  Finally, we recall the following useful density property of the Orlicz spaces. Denoting by $C([a,b])$ the space of all continuous functions $f:[a,b] \to \R$, it turns out that $C([a,b]) \subset L^{\varphi}([a,b])$, and moreover, there holds $C([a,b])$ is dense in $L^{\varphi}([a,b])$ with respect to the topology induced by the modular convergence. 
  
  Similarly to above, denoting respectively by $C_+([a,b])$ and $L^{\varphi}_+([a,b])$ the subspaces of $C([a,b])$ and $L^{\varphi}([a,b])$ of the non-negative functions, it is possible to prove that also $C_+([a,b])$ is dense in $L^{\varphi}_+([a,b])$ with respect to the topology induced by the modular convergence.   

Example of Orlicz spaces are the $L^p([a,b])$ spaces, generated by the convex $\varphi$-function $\varphi(u) :=u^p$, $u \mau 0$, $1 \miu p < +\infty$. Here, we have that $I^{\varphi}[f]=\|f\|_p^p$, i.e., it coincides with usual $L^p$-norm, and further the $\Delta_2$-condition holds, then norm and modular convergence are equivalent.

Other useful examples of Orlicz spaces are the interpolation spaces $L^{\alpha} \log^{\beta} L([a,b])$ (also known as the Zygmund spaces), which are very important in the theory of partial differential equations. The interpolation spaces can be generated by the $\varphi$-functions $\varphi_{\alpha, \beta}:= u^{\alpha}\log^{\beta}(u+e)$, for $\alpha \mau 1$, $\beta>0$, $u \mau 0$. The convex modular functionals corresponding to $\varphi_{\alpha, \beta}$ are: 
$$
I^{\varphi_{\alpha, \beta}}[f] := \int_a^b |f(x)|^{\alpha} \log^{\beta}(e+|f(x)|)\ dx,\ \hskip0.5cm f \in M([a,b]).
$$
Note that, also $\varphi_{\alpha, \beta}$ satisfy the $\Delta_2$-condition. An example of $\varphi$-function for which the $\Delta_2$-condition is not satisfied can be provided by the $\varphi_{\gamma}(u)=e^{u^{\gamma}}-1$, for $\gamma>0$, $u \mau 0$. It is well-known that, $\varphi_{\gamma}$ generate the so-called exponential spaces, and here the modular and norm convergence are not equivalent. The exponential spaces find applications in some embedding theorems between Sobolev spaces. The modular functional corresponding to $\varphi_{\gamma}$ is:
$$
I^{\varphi_{\gamma}}[f] := \int_a^b (e^{|f(x)|^{\gamma}}-1)\ dx,\ \hskip0.5cm f \in M([a,b]).
$$
For further details concerning Orlicz spaces and some their extensions, see e.g., \cite{MU1,MUORL,BAMUVI,BBS1,BCS1,BCS2}. 

%%%%%%%%%%%%%%%%%%%%%%%

\section{Sigmoidal functions and basic assumptions} \label{sec3}

A measurable function $\sigma: \R \to \R$ is called a {\em sigmoidal function} if it satisfies the following conditions:
$$
\lim_{x\to -\infty} \sigma(x)\, =\, 0,   \hskip0.8cm \mbox{and}      \hskip0.8cm \lim_{x\to +\infty} \sigma(x)\, =\, 1.
$$
For any fixed sigmoidal function $\sigma(x)$, we can define the following non-negative density function (\cite{COVI3,COVI4,BAKAVI1,COVI16}):
\be
\phis(x)\ :=\ {1 \over 2}\, [\sigma(x+1)-\sigma(x-1)], \hskip1cm x \in \R.
\ee

From now on, we consider non-decreasing sigmoidal functions $\sigma$, such that, the following conditions hold:
\begin{itemize}
\item[$(\Sigma 1)$] $\sigma(x)-1/2$ is an odd function;
\item[$(\Sigma 2)$] $\phis(x)$ is non-decreasing for $x<0$ and non-increasing for $x \mau 0$;
\item[$(\Sigma 3)$] $\sigma(x)={\mathcal O}(|x|^{-\alpha})$ as $x \to -\infty$, for some $\alpha>0$.
\end{itemize}
Note that, if $\sigma$ belongs to $C^2(\R)$ and it is concave for $x \mau 0$ then $(\Sigma 2)$ turns out to be automatically fulfilled.

 Under the above assumptions on $\sigma(x)$, it is possible to prove that (see \cite{COVI3,COVI4,COVI6}):
\be \label{min}
   \V_{k \in {\mathcal J}_n }\phi_{\sigma}(nx-k)\ \mau\ \phis(2)\ >\ 0,  
\ee
for every sufficiently large $n \in \N^+$, where
$$
{\mathcal J}_n := \left\{ k \in \Z:\  \lceil na \rceil \miu k \miu  \lfloor nb \rfloor - 1 \right\},
$$
$\lfloor \cdot \rfloor$ and $\lceil \cdot \rceil$ denote the integer and the ceiling part of a fixed real number, respectively, and where the symbol $\V$, in the literature, is defined as follows:
\be
\V_{k \in {\mathcal J} } A_k\ :=\ \sup \left\{ A_k \in \R,\ k \in {\mathcal J}\right\},
\ee 
for any set of indexes ${\mathcal J} \subset \Z$. Moreover, it is also possible to prove that, in general, $\phis \in L^1(\R)$, and:
\be
\phi_{\sigma}(x) = \mathcal{O}(|x|^{-\alpha}), \hskip1cm as \hskip1cm x \to 
\pm \infty,
\ee
where $\alpha>0$ is the constant of condition $(\Sigma 3)$ (see \cite{COVI3,COVI4} again).

   Now, we recall the definition of the {\em generalized absolute moment of order $\beta>0$} for the function $\phis(x)$ (see \cite{COVI4}):
\be \label{mom1}
m_{\beta}(\phis)\ :=\ \sup_{x \in \R} \left[ \V_{k \in \Z} \phis(x-k)|x-k|^{\beta} \right].
\ee
The following lemma involving the generalized absolute moments has been proved in \cite{COVI4}.
\begin{lemma} \label{lemma2}
There holds:
$$
m_{\beta}(\phis)\ \miu\ \phis(0)\ \miu\ {1 \over 2},
$$
for every $0 \miu\ \beta \miu \alpha$, where $\alpha>0$ is the constant of condition $(\Sigma 3)$.
\end{lemma}
Examples of sigmoidal functions satisfying the above assumptions are, the logistic function $\sigma_\ell(x):=(1 + e^{-x})^{-1}$ (see e.g., \cite{COSP1,CO1}), the sigmoidal function generated by the hyperbolic tangent $\sigma_h(x):=(\tanh x + 1)/2$, $x \in \R$, (see \cite{COVI7}) and the ramp function (\cite{CHGE}):
$$
\sigma_R(x)\ :=\ \left\{
\begin{array}{l}
0, \hskip2.6cm x < -3/2,\\
\\
x/3 + 1/2, \hskip1cm -3/2 \miu x \miu 3/2, \\
\\
1, \hskip2.6cm x > 3/2,
\end{array}
\right.
$$
and many others, see e.g., \cite{CO1,COGA5,CMV1}. Moreover, an example of discontinuous sigmoidal function for which the present theory holds is given by the following step function:
$$
\sigma_{s}(x)\ :=\
\left\{
\begin{array}{l}
0, \hskip1.6cm x<-2,\\
\\
1/2, \hskip1.2cm -2 \miu x \miu 2,\\ 
\\
1, \hskip1.6cm x > 2.
\end{array}
\right.
$$
In conclusion of this section, we can also note that, all the above examples of sigmoidal functions satisfy assumption $(\Sigma 3)$ for every $\alpha >0$, in view of their fast decay to zero, as $x \to -\infty$. 

%%%%%%%%%%%%%%%%%%%%%%

\section{Convergence results} \label{sec4}

Let us now recall the definition of the operators that will be studied in this paper.
\begin{definition}
Let $f:[a,b] \to \R$ be fixed, and $n \in \N^+$ such that $\lceil na \rceil \miu \lfloor nb \rfloor-1$. The max-product neural network (NN) operators of Kantorovich type activated by $\sigma(x)$ are defined by:
$$
\K_n(f,\, x)\ :=\ \frac{\disp \V_{k \in {\mathcal J}_n}\left[ n \int_{k/n}^{(k+1)/n} f\left( u\right)\, du \right] \phi_{\sigma}(nx-k)}{\disp \V_{k \in {\mathcal J}_n}\phi_{\sigma}(nx-k)}, \hskip0.8cm x \in [a,b].
$$
\end{definition}
Obviously, for $n \in \N^+$ sufficiently large, we always obtain $\lceil na \rceil 
\miu \lfloor nb \rfloor - 1$, and if $\lceil na 
\rceil \miu k \miu \lfloor nb\rfloor - 1$, then $a\miu \frac{k}{n} \miu 
b-\frac{1}{n}$ and $a\miu \frac{k + 1}{n} \miu b$. Moreover, it is easy to see that, the above operators are well-defined, e.g., in case of bounded $f$.

The operators $\K_n$ satisfy the following useful properties (see \cite{COVI3,COVI4}):
\begin{enumerate}

\item if $0\miu f(x) \miu g(x)$, for each $x \in [a,b]$, we have $\K_n(f,x) \miu \K_n(g,x)$, for every $x \in [a,b]$;

\item $\K_n(f+g,\, x) \miu \K_n(f,x)+\K_n(g,x)$, $x \in [a,b]$, $f$, $g \mau 0$, i.e., $\K_n$ are sub-additives (or sub-linear) operators;

\item $\left| \K_n(f, x)-\K_n(g, x) \right| \miu \K_n(|f-g|,\, x)$, $x \in [a,b]$, $f$, $g \mau 0$;

\item for each $\lambda>0$, $\K_n(\lambda f, x) = \lambda \K(f, x)$, $x \in [a,b]$, $f$, $g \mau 0$;

\item denoting by ${\bf 1}$ the unitary constant function on $[a,b]$, it turns out that:
$$
\K_n({\bf 1},\, x)\ =\ 1, \hskip1cm x \in [a,b].
$$
\end{enumerate}

Further, in order to prove the desired modular convergence result in Orlicz space, we need to recall the following pointwise and uniform convergence theorem.
\begin{theorem} \label{th1}
Let $f:[a,b] \to \R^+_0$ be a given bounded function, continuous at $x \in [a,b]$. Then,
$$
\lim_{w \to +\infty} \K_n(f, x)\ =\ f(x).
$$
Moreover, in case of $f \in C_+([a,b])$, we have:
$$
\lim_{w \to +\infty} \|\K_n(f, \cdot)\ -\ f(\cdot)\|_{\infty}\ =\ 0.
$$
\end{theorem}

From now on, we always consider a fixed convex $\varphi$-function $\varphi$, as defined in Section \ref{sec2}. Now, we can test the (Luxemburg) norm convergence of the operators $\K_n$, when continuous functions are considered.
\begin{theorem} \label{th2}
Let $f \in C_+([a,b])$ be fixed. Then, for every $\lambda>0$:
$$
\lim_{n \to +\infty} I^{\varphi}\left[ \lambda\left(   \K_n(f, \cdot) - f(\cdot) \right)\right]\ =\ 0.
$$
\end{theorem}
\begin{proof}
Let $\ep>0$ be fixed. Thus, for every fixed $\lambda>0$, using the convexity of $\varphi$, and in view of Theorem \ref{th1}, we have:
$$
I^{\varphi}\left[ \lambda\left( \K_n(f, \cdot) - f(\cdot) \right)\right]\ =\ \int_a^b \varphi\left(\lambda \left| \K_n(f, x) - f(x) \right| \right)\, dx
$$
$$
\miu\ \int_a^b \varphi\left(\lambda \| \K_n(f, \cdot) - f(\cdot) \|_{\infty} \right)\, dx\ \miu\ \int_a^b \varphi\left(\lambda\, \ep \right)\, dx\ \miu\ \ep\, \varphi(\lambda) (b-a),
$$
for $n \in \N^+$ sufficiently large, thus the proof follows by the arbitrariness of $\ep>0$.
\end{proof}

Now, we need to prove a modular inequality for the above operators.
\begin{theorem} \label{th3}
For every $f$, $g \in L^{\varphi}_+([a,b])$, and $\lambda>0$, it turns out that:
$$
I^{\varphi}\left[ \lambda\left( \K_n(f, \cdot) - \K_n(g, \cdot) \right)\right]\ \miu\ \|\phis\|_1\, I^{\varphi}\left[ \phis(2)^{-1}\lambda\,  \left(f(\cdot)-g(\cdot)\right) \right],
$$
for $n \in \N^+$ sufficiently large.
\end{theorem}
\begin{proof}
Using the property 3. of the operators $\K_n$, and condition (\ref{min}), we can write what follows:
$$
I^{\varphi}\left[ \lambda\left( \K_n(f, \cdot) - \K_n(g, \cdot) \right)\right]\ =\ \int_a^b \varphi\left( \lambda \left| \K_n(f, x) - \K_n(g, x)   \right|   \right)\, dx
$$
$$
\hskip-6.6cm \miu\ \int_a^b \varphi\left( \lambda\, \K_n(|f-g|, x) \right)\, dx
$$
$$
\miu\ \int_a^b \varphi\left( \lambda\, \phis(2)^{-1} \V_{k \in {\mathcal J}_n} \left[n \int_{k/n}^{(k+1)/n} |f(u)-g(u)|\, du \right] \phis(nx-k) \right)\, dx.
$$
Now, it is easy to observe that, the following equality holds:
\be \label{scambio}
\varphi\left( \V_{k \in {\mathcal J}} A_k  \right)\ =\ \V_{k \in {\mathcal J}}\varphi\left( A_k  \right),
\ee
for any finite set of indexes ${\mathcal J} \subset \Z$, since $\varphi$ is non-decreasing. Thus, using (\ref{scambio}), the convexity of $\varphi$ with the fact that by Lemma \ref{lemma2}, $\phis(nx-k) \miu 1/2$, for every $k \in {\mathcal J}_n$, and the Jensen's inequality, we obtain:
$$
\hskip-5cm I^{\varphi}\left[ \lambda\left( \K_n(f, \cdot) - \K_n(g, \cdot) \right)\right]\ 
$$
$$
\miu\ \int_a^b \V_{k \in {\mathcal J}_n}  \varphi\left( \lambda\, \phis(2)^{-1} \left[n \int_{k/n}^{(k+1)/n} |f(u)-g(u)|\, du \right] \phis(nx-k) \right)\, dx
$$
$$
\miu\ \int_a^b \V_{k \in {\mathcal J}_n} \phis(nx-k)\,  \varphi\left( \lambda\, \phis(2)^{-1} \left[n \int_{k/n}^{(k+1)/n} |f(u)-g(u)|\, du \right] \right)\, dx
$$
$$
\hskip-0.2cm \miu\ \int_a^b dx \left\{\V_{k \in {\mathcal J}_n} \phis(nx-k)\, n \int_{k/n}^{(k+1)/n}  \varphi\left( \lambda\, \phis(2)^{-1} |f(u)-g(u)|  \right)\, du \right\}
$$
$$
\hskip-1.2cm \miu\ \int_a^b dx \left\{ \V_{k \in {\mathcal J}_n} \phis(nx-k)\, n \int_{a}^{b}  \varphi\left( \lambda\, \phis(2)^{-1} |f(u)-g(u)|  \right)\, du\right\}
$$
$$
\hskip-2.8cm =\ I^{\varphi}\left[ \lambda\, \phis(2)^{-1} \left(f(\cdot)-g(\cdot)\right) \right]\, \V_{k \in {\mathcal J}_n}  n\, \int_a^b  \phis(nx-k)\, dx.
$$
Now, using the change of variable $y=nx$, and the property:
$$
\|\phis\|_1\ =\ \int_{\R} \phis(y)\ dy\ =\ \int_{\R} \phis(y - k)\ dy,
$$
for every $k \in \Z$, we finally obtain:
$$
\hskip-5cm I^{\varphi}\left[ \lambda\left( \K_n(f, \cdot) - \K_n(g, \cdot) \right)\right]\ 
$$
$$
\miu\ I^{\varphi}\left[ \lambda\, \phis(2)^{-1} \left(f(\cdot)-g(\cdot)\right) \right]\, \left[ \V_{k \in {\mathcal J}_n}  \int_\R \phis(y-k)\, dy \right]
$$
$$
\hskip-2.7cm =\ I^{\varphi}\left[ \phis(2)^{-1} \lambda\, \left(f(\cdot)-g(\cdot)\right) \right]\, \|\phis\|_1,
$$
for every $n \in \N^+$ sufficiently large.
\end{proof}
Now, we are able to prove the main result of this section, i.e., a modular convergence theorem for the Kantorovich max-product NN operators.
\begin{theorem} \label{th4}
Let $f \in L^{\varphi}_+([a,b])$ be fixed. Then there exists $\lambda>0$ such that:
$$
\lim_{n \to +\infty} I^{\varphi}\left[\lambda \left( \K_n(f, \cdot) - f(\cdot)  \right)\right]\ =\ 0.
$$
\end{theorem}
\begin{proof}
Let $\ep>0$ be fixed. Since $C_+([a,b])$ is modularly dense in $L^{\varphi}_+([a,b])$, there exists $\bar \lambda >0$ and $g \in C_+([a,b])$ such that:
\be \label{densitaa}
I^{\varphi}\left[\bar \lambda \left( f(\cdot) - g(\cdot)  \right)\right]\ <\ \ep.
\ee
Now, choosing $\lambda >0$ such that:
$$
\max \left\{ \phis(2)^{-1}3\, \lambda,\ 3\, \lambda \right\}\, \miu\ \bar \lambda,
$$
by the property of $I^\varphi$, using Theorem \ref{th3}, and the convexity of $\varphi$, we can write what follows:
\vskip0.1cm
$$
\hskip-7cm I^{\varphi}\left[\lambda \left( \K_n(f, \cdot) - f(\cdot)  \right)\right]\ 
$$
$$
\hskip0.8cm \miu\ {1 \over 3} \left\{I^{\varphi}\left[3 \lambda \left( \K_n(f, \cdot) - \K_n(g, \cdot)  \right)\right] + I^{\varphi}\left[3 \lambda \left( \K_n(g, \cdot) - g(\cdot)  \right)\right] \right.
$$
$$
\hskip-7.2cm +\left. \ I^{\varphi}\left[3 \lambda \left( g(\cdot) - f(\cdot)  \right)\right]  \right\}
$$
$$
\miu\ \|\phis\|_1\, I^{\varphi}\left[ \phis(2)^{-1} 3 \lambda\,  \left|f(\cdot)-g(\cdot)\right|\right]+ I^{\varphi}\left[ 3 \lambda \left( \K_n(g, \cdot) - g(\cdot)  \right)\right]  
$$
$$
\hskip-7.4cm +\ I^{\varphi}\left[3 \lambda \left( g(\cdot) - f(\cdot)  \right)\right]
$$
$$
\hskip-0.2cm \miu\ \left( \|\phis\|_1 + 1\right)I^{\varphi}\left[\bar \lambda \left( g(\cdot) - f(\cdot)  \right)\right] + I^{\varphi}\left[\bar \lambda \left( \K_n(g, \cdot) - g(\cdot)  \right)\right], 
$$
for every $n \in \N^+$ sufficiently large. Now, using (\ref{densitaa}):
$$
I^{\varphi}\left[\lambda \left( \K_n(f, \cdot) - f(\cdot)  \right)\right]\, \miu\ \left( \|\phis\|_1 + 1\right)\, \ep\ + I^{\varphi}\left[\bar \lambda \left( \K_n(g, \cdot) - g(\cdot)  \right)\right],
$$
and since Theorem \ref{th2} holds, we get:
$$
I^{\varphi}\left[\bar \lambda \left( \K_n(g, \cdot) - g(\cdot)  \right)\right]\ <\ \ep,
$$
for every $n \in \N^+$ sufficiently large. In conclusion, we finally obtain:
$$
I^{\varphi}\left[\lambda \left( \K_n(f, \cdot) - f(\cdot)  \right)\right]\, \miu\ \left( \|\phis\|_1 + 2\right)\, \ep,
$$
for every $n \in \N^+$ sufficiently large. Thus the proof follows by the arbitrariness of $\ep>0$.
\end{proof}
Note that, all the above convergence results can be extended to the case of not-necessarily non-negative functions. Indeed, for any fixed $f:[a,b] \to \R$, if $\inf_{x \in [a,b]} f(x) =: c <0$, we can consider the sequences $(\K_n(f-c, \cdot)+c)_{n \in \N+}$ that it turns out to be convergent to $f$ (in all the above considered senses, provided $f$ in suitable spaces). For more details, see e.g., \cite{COGA3,COGA4,COVI3}.

%%%%%%%%%%%%%%%%%%%%%%%

\section{Quantitative estimates in Orlicz spaces} \label{sec5}

In \cite{COVI7}, quantitative estimates for the operators $\K_n$ have been proved for functions belonging to $C_+([a,b])$, exploiting the modulus of continuity of the function to be approximated, and for functions belonging to $L^p_+([a,b])$, $1 \miu p < +\infty$, using the well-known definition of the K-functionals introduced by Peetre \cite{PE1}. Here, we prove quantitative estimates in the general setting of $L^{\varphi}_+([a,b])$, with $\varphi$ convex. For this purpose, we need to introduce suitable K-functionals in this setting.

  For any fixed $f \in L^\varphi_+([a,b])$, there exist $\lambda>0$ such that, the K-functional of $f$ can be defined as follows:
\be
{\mathcal K}(f, \lambda, \delta)_\varphi\ :=\ \inf_{g \in C^1_+([a,b])}\ \left\{ I^{\varphi}\left[ \lambda\, (f(\cdot)-g(\cdot) ) \right]\ +\ \delta\, \varphi\left(\|g'\|_{\infty}\right) \right\},
\ee
$\delta > 0$, where $C^1_+([a,b])$ is the space of non-negative differentiable functions with continuous derivative. Note that, $C^1_+([a,b]) \subset E^{\varphi}([a,b]) \subset L^\varphi_+([a,b])$, then the definition of ${\mathcal K}(f, \lambda, \delta)_\varphi$ is well-posed, i.e., for any fixed $f \in L^\varphi_+([a,b])$, there exists $\lambda>0$ such that:
$$
{\mathcal K}(f, \lambda, \delta)_\varphi\ \miu\ I^{\varphi}\left[ 2\lambda\, f(\cdot) \right]\ +\ I^{\varphi}\left[ 2\lambda\, g(\cdot) \right]\ +\ \delta \varphi\left(\|g'\|_{\infty}\right)\ <\ +\infty,
$$
for every $g \in C^1_+([a,b])$.

  The K-functional ${\mathcal K}(f, \lambda, \delta)_\varphi$, provides information about the approximation properties of $f$: the inequality ${\mathcal K}(f, \lambda,\delta)_\varphi<\ep$, for some $\lambda>0$ and $\delta>0$, implies that $f$ can be approximated with an error $I^{\varphi}\left[ \lambda\, (f(\cdot)-g(\cdot) ) \right]<\ep$, and $g$ belongs to $C^1_+([a,b])$, whose sup-norm $\|\cdot\|_{\infty}$ of $g'$ (more precisely $\varphi\left( \|g'\|_{\infty} \right)$) is not too large. 
  
  We can prove the following.
\begin{theorem} \label{th5.1}
Let $\sigma(x)$ be a sigmoidal function which satisfies condition $(\Sigma 3)$ with $\alpha \mau 1$. Moreover, let $f \in L^\varphi_+([a,b])$, be fixed. Then, there exist $\lambda_0>0$ and $\lambda_1 >0$, such that:
$$
I^{\varphi}\left[ \lambda_1 (K_n(f, \cdot)-f(\cdot))\right]\ \miu\ A_1\, {\mathcal K}\left(f, \lambda_0, A_2\, n^{-1}\right)_\varphi,
$$
for every sufficiently large $n \in \N^+$, where:
$$
A_1\ :=\ \|\phis\|_1 + 1, \quad
and \quad
A_2\ :=\ {3\, \lambda_0 (b-a) \over 4\, \phis(2) \left( \|\phis\|_1 + 1 \right)}.
$$
\end{theorem}
\begin{proof}
First of all, since $f \in L^\varphi_+([a,b])$, there exists $\lambda_0>0$ such that $I^\varphi[\lambda_0 f]<+\infty$. Then, we choose $\lambda_1 > 0$ such that:
$$
\max\left\{ 3 \lambda_1,\ \phis(2)^{-1}3\lambda_1    \right\}\ \miu\ \lambda_0.
$$
Let now $g \in C^1_+([a,b])$ be fixed. Using the properties of $I^\varphi$, and Theorem \ref{th3}, can write what follows:
$$
\hskip-7.6cm I^\varphi\left[ \lambda_1\, (\K_n(f, \cdot)-f(\cdot)) \right]
$$
$$
\miu I^\varphi\left[ 3\lambda_1 \left(\K_n(f, \cdot)-\K_n(g, \cdot)\right) \right]+ I^\varphi\left[ 3\lambda_1 \left(\K_n(g, \cdot) - g(\cdot)\right) \right] + I^\varphi\left[ 3\lambda_1 (g(\cdot)-f(\cdot))\right]
$$ 
$$
\hskip-1.8cm \miu\ \left( \|\phis\|_1 + 1 \right) I^\varphi\left[ \lambda_0\, (g(\cdot)-f(\cdot))\right] + I^\varphi\left[ \lambda_0\, \left(\K_n(g, \cdot) - g(\cdot)\right) \right]
$$
for sufficiently large $n \in \N^+$. Now, for every fixed $x \in [a,b]$, we can define the function $h_x(t):=g(x)$, $t \in [a,b]$, then using properties 4. and 3. of $\K_n$, we obtain:
$$
\left|\K_n(g, x) - g(x)\right|\ =\ \left| \K_n(g, x) - \K_n(h_x, x) \right|\ \miu\ \K_n(|g - h_x|, x).
$$
Since $g \in C^1_+([a,b])$, it is easy to observe that:
$$
|g(t) - h_x(t)|\ =\ |g(t) - g(x)|\ \miu\ \|g'\|_{\infty}\, |t-x|\ =:\ \|g'\|_{\infty} \Psi_x(t),
$$
for every $t \in [a,b]$. Thus, using the properties 1. and 4. of $\K_n$, we obtain:
$$
\hskip-5cm \K_n(|g - h_x|, x)\ \miu\ \|g'\|_{\infty}\, \K_n(\Psi_x, x)
$$
$$
\hskip-1.5cm =\ \|g'\|_{\infty}\, \frac{\disp \V_{k \in {\mathcal J}_n }\left[ n \int_{k/n}^{(k+1)/n} |x - u|\, du \right] \phi_{\sigma}(nx-k)}{\disp \V_{k \in {\mathcal J}_n }\phi_{\sigma}(nx-k)}
$$
$$
\hskip-0.8cm \miu\ \frac{\|g'\|_{\infty}}{\phis(2)}\left[  \V_{k \in {\mathcal J}_n }\left[ n \int_{k/n}^{(k+1)/n} |x - k/n|\, du \right] \phi_{\sigma}(nx-k)  \right.
$$
$$
\left. \hskip-2.2cm + \V_{k \in {\mathcal J}_n }\left[ n \int_{k/n}^{(k+1)/n} |k/n - u|\, du \right] \phi_{\sigma}(nx-k) \right]
$$
$$
\hskip0.6cm =\ \frac{\|g'\|_{\infty}}{\phis(2)}\, \left[ \V_{k \in {\mathcal J}_n }|x - k/n|\, \phi_{\sigma}(nx-k) +\ {1 \over 2} \V_{k \in {\mathcal J}_n } {1 \over n} \phi_{\sigma}(nx-k) \right]
$$
$$
\hskip0.5cm \miu\ n^{-1}\, \frac{\|g'\|_{\infty}}{\phis(2)}\, \left[ m_1(\phis)\ +\ {1 \over 2}\, m_0(\phis) \right] \miu\ n^{-1}\, \frac{3\, \|g'\|_{\infty}}{4\, \phis(2)} <+\infty,
$$
in view of Lemma \ref{lemma2}, and since $\alpha \mau 1$. Then, using the convexity of $\varphi$, we have:
$$
I^\varphi\left[ \lambda_0\, \left(\K_n(g, \cdot) - g(\cdot)\right) \right]\ \miu\ {3\, \lambda_0 \over 4\, \phis(2)}\, n^{-1} \varphi\left( \|g'\|_{\infty}\right)\, (b-a),
$$
for every sufficiently large $n \in \N^+$. In conclusion, rearranging all the above inequalities we obtain:
\vskip0.01cm
$$
\hskip-6.8cm I^\varphi\left[ \lambda_1\, (\K_n(f, \cdot)-f(\cdot)) \right]
$$
$$
\miu\  \left( \|\phis\|_1 + 1 \right) I^\varphi\left[ \lambda_0\, (g(\cdot)-f(\cdot))\right] + {3\, \lambda_0 (b-a) \over 4\, \phis(2)}\, n^{-1} \varphi\left( \|g'\|_{\infty}\right),
$$
and for the arbitrariness of $g$, we finally get:
$$
I^\varphi\left[ \lambda_1\, (\K_n(f, \cdot)-f(\cdot)) \right]\ \miu\ A_1\, {\mathcal K}\left(f, \lambda_0, A_2\, n^{-1}\right)_\varphi,
$$
for every $n \in \N^+$ sufficiently large, where:
$$
A_1 := \left( \|\phis\|_1 + 1 \right), \quad \quad A_2 := {3\, \lambda_0 (b-a) \over 4\, \phis(2) \left( \|\phis\|_1 + 1 \right)}.
$$
\end{proof}
%

%%%%%%%%%%%%%%%%%%%%%%%%%%%

\vskip0.2cm

\section*{Acknowledgments}

The authors are members of the Gruppo  
Nazionale per l'Analisi Matematica, la Probabilità e le loro  
Applicazioni (GNAMPA) of the Istituto Nazionale di Alta Matematica (INdAM). 

\noindent The authors are partially supported by the "Department of Mathematics and Computer Science" of the University of Perugia (Italy). Moreover, the first author of the paper has been partially supported within the 2017 GNAMPA-INdAM Project ``Approssimazione con operatori discreti e problemi di minimo per funzionali del calcolo delle variazioni con applicazioni all'imaging''.

\noindent This is a post-peer-review, pre-copyedit version of an article published in  Results in Mathematics. The final authenticated version is available online at: http://dx.doi.org/0.1007/s00025-018-0799-4

\vskip0.1cm
%
%

%%%%%%
\end{document}